\documentclass{amsart}

\usepackage{amsmath, amssymb, amsthm, amsfonts, amscd, mathrsfs}
\usepackage[T1]{fontenc}
\usepackage[utf8x,applemac]{inputenc}
\usepackage{graphicx,pict2e}
\usepackage[colorlinks=true]{hyperref}
\usepackage{tikz}
\usepackage[shortlabels]{enumitem}

\newcommand{\R}{\mathbb{R}}
\newcommand{\Q}{\mathbb{Q}}

\newcommand{\N}{\mathbb{N}}
\newcommand{\calP}{\mathcal{P}}

\newcommand{\ve}{\varepsilon}

\newcommand{\abs}[1]{\left\lvert#1\right\rvert}

\newcommand{\pical}{\mathcal{P}}
\newcommand{\lcal}{\mathcal{L}}
\newcommand{\haus}{\mathcal H}
\newcommand{\huno}{\haus^1}
\newcommand{\res}{\mathop{\hbox{\vrule height 7pt width .5pt depth 0pt \vrule height .5pt width 6pt depth 0pt}}\nolimits}

\newcommand{\restr}{\textrm{\scriptsize{|}}}

\DeclareMathOperator*{\spt}{supp}

\DeclareMathOperator*{\argmax}{arg\,max}

\DeclareMathOperator{\isot}{isot}

\numberwithin{equation}{section}

\theoremstyle{plain}
\newtheorem{thm}{Theorem}[section]
\newtheorem{lemma}[thm]{Lemma}
\newtheorem{prop}[thm]{Proposition}
\newtheorem{cor}[thm]{Corollary}
\newtheorem*{lem*}{Lemme}
\newtheorem*{mainthm*}{Main Theorem}

\theoremstyle{definition}
\newtheorem{defi}[thm]{Definition}

\theoremstyle{remark}
\newtheorem{rem}[thm]{Remark}

\tikzset { domaine/.style 2 args={domain=#1:#2} }
\tikzset{
xmin/.store in=\xmin, xmin/.default=-3, xmin=-3,
xmax/.store in=\xmax, xmax/.default=3, xmax=3,
ymin/.store in=\ymin, ymin/.default=-3, ymin=-3,
ymax/.store in=\ymax, ymax/.default=3, ymax=3,
}
\begin{document}

\title[Optimal transport for concave costs]{Full characterization of optimal transport plans for concave costs}
\author{Paul Pegon, Davide Piazzoli, Filippo Santambrogio}
\date{2015}

\begin{abstract}

This paper slightly improves a classical result by Gangbo and McCann (1996) about the structure of optimal transport plans for costs that are strictly concave and increasing functions of the Euclidean distance.
 Since the main difficulty for proving the existence of an optimal map comes from the possible singularity of the cost at $0$, everything is quite easy if the supports of the two measures are disjoint;
 Gangbo and McCann proved the result under the assumption $\mu(\mathrm{supp}(\nu))=0$;
 in this paper we replace this assumption with the fact that the two measures are singular to each other. In this case it is possible to prove the existence of an optimal transport map, provided the starting measure $\mu$ does not give mass to small sets (i.e. $(d-1)$-rectifiable sets). When the measures are not singular the optimal transport plan decomposes into two parts, one concentrated on the diagonal and the other being a transport map between mutually singular measures.
\end{abstract}

\maketitle

\subjclass{{\bf MSC 2010} :}  49J45, 49K21, 49Q20, 28A75


\tableofcontents

\section{Introduction}

Optimal transport is nowadays a very powerful and widely studied theory for many applications and connections with other pieces of mathematics. The minimization problem 
$$(M)\quad\min \left\{\quad\int c(x,T(x))d\mu(x)\;:\;T_\#\mu=\nu\right\}$$
proposed by Monge in 1781 (see \cite{Mon}) has been deeply understood thanks to the relaxation proposed by Kantorovich in \cite{Kan} in the form of a linear programming problem

\begin{equation}\label{kantorovich}
(K)\quad\min\left\{\int_{\R^d\times\R^d}\!\!c\,\,d\gamma\;:\;\gamma\in\Pi(\mu,\nu)\right\}.
\end{equation}

Here $\mu$ and $\nu$ are two probability measures on $\R^d$ and $c$ a cost function $c:\R^d\times\R^d\to\R$.
The set $\Pi(\mu,\nu)$ is the set of the so-called {\it transport plans}, i.e. $\Pi(\mu,\nu)=\{\gamma\in\pical(\R^d\times\R^d):\,(\pi_x)_{\#}\gamma=\mu,\,(\pi_y)_{\#}\gamma=\nu\}$
where $\pi_x$ and $\pi_y$ are the two projections of $\R^d\times\R^d$
onto $\R^d$. These probability measures over $\R^d\times\R^d$ are an alternative way to describe the displacement of the particles of $\mu$: 
instead of saying, for each $x$, which is the destination $T(x)$ of the particle originally located at $x$, we say for each pair $(x,y)$ how many particles go from $x$ to $y$. 
It is clear that this description allows for more general movements, since from a single point $x$ particles can a priori move to different destinations $y$. 
If multiple destinations really occur, then this movement cannot be described through a map $T$.

 The Kantorovich problem is interesting in itself and carries many of the features of Monge's one.
Since it can be rigorously proven to be its relaxation in the sense of l.s.c. envelops, the minimal value of the two problems is the same, provided $c$ is continuous. 
For many applications, dealing with the optimum of (K) is enough. 
Yet, a very classical question is whether the optimizer $\gamma$ of (K) is such that for almost every $x$ only one point $y$ can be such that $(x,y)\in \spt(\gamma)$. 
In this case $\gamma$ will be of the form $(id\times T)_{\#}\mu$ and will provide an optimal transport map for (M).

For several different applications, from fluid mechanics to differential geometry, the case which has been studied the most is the quadratic one, $c(x,y)=|x-y|^2$, first solved by Brenier in \cite{Brenier91}. 
Other costs which are strictly convex functions of $x-y$, for instance all the powers $|x-y|^p$, $p>1$, can be dealt with in a similar way. 
Next Section will give a general strategy to prove the existence of a transport map which fits very well this case. 
The limit case $c(x,y)=|x-y|$, which was by the way the original interest of Monge, has also received much attention.

Yet, another very natural case is that of concave costs, more precisely $c(x,y)=\ell(|x-y|)$ where $\ell:\R_+\to\R_+$ is a strictly concave and increasing function.
From the economical and modelization point of view, this is the most natural choice: moving a mass has a cost which is proportionally less if the distance increases, as everybody can notice from travel fares. 
In many practical cases, moving two masses on a distance $d$ each is more expensive than moving one at distance $2d$ and keeping at rest the other.
The typical example is the power cost $|x-y|^\alpha$, $\alpha<1$. Notice that all these costs satisfy the triangle inequality and are thus distances on $\R^d$. 
Among the other interesting features of these costs,  let us mention two. From the theoretical point of view, there is the fact that all power costs $|x-y|^\alpha$ with $\alpha<1$ satisfy Ma-Trudinger-Wang assumption for regularity, see \cite{MTW}. From the computational point of view, the subadditivity properties of these costs allow for some efficient algorithms using local indicators, at least in the one dimensional discrete case, see \cite{DelSalSob}.

Moreover, under strict convexity assumptions, these costs satisfy a strict triangle inequality (see Lemma \ref{strict dist}).
This last fact implies (see Theorem \ref{stay at rest}, but it is a classical fact) that the common mass between $\mu$ and $\nu$ must stay at rest. 
This gives a first constraint on how to build optimal plans $\gamma$: 
look at $\mu$ and $\nu$, take the common part $\mu\wedge\nu$, leave it on place, subtract it from the rest, and then build an optimal transport between the two remainders, which will have no mass in common. 
Notice that when the cost $c$ is linear in the Euclidean distance, then the common mass {\it may} stay at rest but is not forced to do so 
(the very well known example is the transport from $\mu=\lcal^1\res[0,1]$ and $\nu=\lcal^1\res[\frac 12,\frac 32]$, where both $T(x)=x+\frac 12$ and $T(x)=x+1$ on $[0,\frac 12]$ and $T(x)=x$ on $]\frac 12,1]$ are optimal); 
on the contrary, when the cost is strictly convex in the Euclidean distance, in general the common mass {\it does not} stay at rest (in the previous example only the translation is optimal for $c(x,y)=|x-y|^p$, $p>1$).
Notice that the fact that the common mass stays at rest implies that in general there is no optimal map $T$, 
since whenever there is a set $A$ with $\mu(A)>(\mu\wedge\nu)(A)=\nu(A)$ then almost all the points of $A$ must have two images: themselves, and another point outside $A$.

Yet, this suggests to study the case where $\mu$ and $\nu$ are mutually singular, and the best one can do would be proving the existence of an optimal map in this case. 
This is a good point, since the singularity of the function $(x,y)\mapsto \ell(|x-y|)$ is mainly concentrated on the diagonal $\{x=y\}$ (look at the example $|x-y|^\alpha$), 
and when the two measures have no common mass almost no point $x$ is transported to $y=x$. Yet, exploiting this fact needs some attention.

First, a typical assumption on the starting measure is required: we need to suppose that $\mu$ does not give mass to $(d\!-\!1)-$rectifiable sets. This is standard and common with other costs, such as the quadratic one. 
From the technical point of view, this is needed in order to guarantee $\mu-$a.e. differentiability of the Kantorovich potential, 
and counter-examples are known without this assumption (see next Section both for Kantorovich potentials and for counter-examples).

Hence, if we add this assumption on $\mu$, the easiest case is when $\mu$ and $\nu$ have disjoint supports, since in this case there is a lower bound on $|x-y|$ and this allow to stay away from the singularity. 
Yet, $\spt\mu\cap\spt\nu=\emptyset$ is too restrictive, 
since even in the case where $\mu$ and $\nu$ have smooth densities $f$ and $g$ it may happen that, after subtracting the common mass, the two supports meet on the region $\{f=g\}$. 

The problem has actually been solved in one of the first papers about optimal transportation, written by Gangbo and McCann in 1996, \cite{GaMc}, where they choose the slightly less restrictive assumption $\mu(\spt(\nu))=0$. 
This assumption covers the example above of two continuous densities, but does not cover other cases such as $\mu$ being the Lebesgue measure on a bounded domain $\Omega$ and $\nu$ being an atomic measure with an atom at each rational point, 
or other examples that one can build with fully supported absolutely continuous measures concentrated on disjoint sets $A$ and $\Omega\setminus A$. The present paper completes the proof by Gangbo and McCann, making use of recent ideas on optimal transportation to tackle the general case; 
i.e. we solve the problem under the only assumption that $\mu$ and $\nu$ are singular to each other (and that $\mu$ does not give mass to ``small'' (i.e. $(d\!-\!1)-$rectifiable) sets). From the case of mutually singular measures we can deduce how to deal with the case of measures with a common mass. The title of the paper exactly refers to this fact: by ``characterization of optimal transport plans'' we mean ``understanding their structure, composed by a diagonal part and a transport map out of the diagonal'', the word ``full'' stands for the fact that we arrive to the mimimal set of assumptions with respect to previous works, and by ``concave costs'' we indeed mean ``costs which are strictly concave and increasing functions of the euclidean distance''.

In Section 2 we present the main tools that we need. Section 2.1 is devoted to well-known facts from optimal transport theory: we recall the usual strategy to prove the existence of an optimal $T$ based on the Kantorovich potential 
and, after proving that the common mass stays at rest via a $c-$cyclical monotonicity argument, we adapt them to the concave case. 
Section 2.2 recalls the notion and some properties of approximate gradients, which we will use in the following Section.

In Section 3 we start generalizing Gangbo and McCann's result, in the case where $\mu$ is absolutely continuous. 
In this case we use the fact that almost every point $x$ is sent to a different point $y\neq x$ together with density points argument in order to prove that the Kantorovich potential is approximately differentiable almost everywhere, 
according to the notions that we presented in Section 2.2. 
Notice that this strategy is very much linked to many arguments recently used in optimal transportation in \cite{ChaDeP09,ChaDePJuu}, 
where the existence of an optimal map is proven by restricting the transport plan $\gamma$ to a suitable set of Lebesgue points which is $c-$cyclically monotone. 
Here we do not need to address explicitly such a construction but the idea is very much similar. The result we obtain in this section is not contained in \cite{GaMc} but it does not contain their result neither.

It is in Section 4 that we consider the most general case of an arbitrary measure $\mu$ not giving mass to small sets.
The approximate differentiability of the potential in Section 3 was based on Lebesgue points arguments which require to be adapted to this new framework. 
This is why we present an interesting Geometric Measure Theory Lemma (Lemma \ref{Paul's}) which states that, whenever $\mu$ does not give mass to small sets, 
then $\mu-$almost every point $x$ is such that every cone with vertex at $x$, of arbitrary size, direction and opening, has positive mass for $\mu$. 
This lemma is not new, but it is surprisingly not so well-known, at least in this very formulation, in the geometric measure theory community (a weaker version of this lemma is actually contained in the classical book \cite{Fed}). On the contrary, it starts being popular in the optimal transport community (see \cite{ChaDeP DCDS}, where the authors prove it and say that it has been presented to them by T. Rajala), which is strange if we think that it is really a GMT statement. Yet, even if not concerned directly with optimal transport, it has been popularized thanks to its applications in optimal transport theory. 

For the sake of completeness, we present the proof of this lemma in Section 4.1, even if we stress that the contribution of the paper is not in such a proof; but in the applications of this result to the differentiability of the Kantorovich potential that we face. This is what we do in Section 4.2, where we define an ad-hoc notion of gradient for the Kantorovich potential, by using this density result to prove that it is well-defined $\mu-$a.e. and that it satisfies all the properties we need. Finally, we prove that the optimal $\gamma$ is concentrated on a graph.

In this way we can now state the main theorem of this paper. Define $\mu\wedge\nu$ as the maximal positive measure which is both less or equal than $\mu$ and than $\nu$, and $(\mu-\nu)_+=\mu-\mu\wedge\nu$, 
so that the two measures $\mu$ and $\nu$ uniquely decompose into a common part $\mu\wedge\nu$ and two mutually singular parts $(\mu-\nu)_+$ and $(\nu-\mu)_+$.
\begin{mainthm*}
Suppose that $\mu$ and $\nu$ are probability measures on $\R^d$ such that $(\mu-\nu)_+$ gives no mass to all $(d\!-\!1)-$rectifiable sets, and take as a cost the function $c(x,y)=\ell(|x-y|)$, for $\ell:\R_+\to \R_+$ strictly concave and increasing. 
Then there exists a unique optimal transport plan $\gamma$, and it has the form $(id,id)_\#(\mu\wedge \nu)+(id,T)_\#(\mu-\nu)_+$.
\end{mainthm*}
This theorem is obtained as a corollary of Theorem \ref{complete} which concerns measures with no common mass, 
whereas Theorem \ref{simplified} is a simplified version of the same statement, where the assumption that $\mu$ is absolutely continuous plays an important role.

The paper ends with an appendix in two parts. One explains that, differently from convex costs, in the case of concave costs translations are never optimal, while the second one presents a discussion about the possibility of defining a sort of approximate gradient adapted to the measure $\mu$.
 %, the third presents an example of two disjoint sets with full support, which is useful several times in the paper.

\section{Tools}

\subsection{General facts on optimal transportation}\label{sec:gene_facts}

We start the preliminaries of this paper with some important and well-known facts about Kantrovich linear programming problem
\begin{equation}\label{kantorovich}\tag{K}
\min\left\{\int_{\R^d\times\R^d}\!\!c\,\,d\gamma\;:\;\gamma\in\Pi(\mu,\nu)\right\},
\end{equation}
Here the cost function $c:\R^d\times\R^d\to\R_+$ is supposed to be continuous. We can suppose the supports of the two measures to be compact, for simplicity, but we do not need it (yet, in this case it is better to suppose $c$ to be uniformly continuous). 

We need to underline some main aspects of this problem (K). First, as any linear programming problem, it admits a dual problem, which reads 
\begin{equation}\label{dual}\tag{D}
\max\left\{\int\!\phi \,d\mu+\!\!\int\!\psi \,d\nu\;:\;\,\phi(x)\!+\!\psi(y)\leq c(x,y) \mbox{ for all }(x,y)\in\R^d\!\times\!\R^d\right\},
\end{equation}
where the supremum is computed over all pairs $(\phi,\psi)$ of continuous functions on $\R^d$.
We refer to \cite{villani} for this duality relation. Here are some of the properties of (K), (D) and the connection between them:
\begin{itemize}
\item (K) admits at least a solution $\gamma$, called {\it optimal transport plan};
\item (D) also has a solution $(\phi,\psi)$;
\item The functions $(\phi,\psi)$ are such that 
$$\psi(y)=\inf_x c(x,y)-\phi(x)\;\mbox{ and }\; \phi(x)=\inf_y c(x,y)-\psi(y)$$ 
(we say that they are conjugate to each other and $\phi=\psi^c$ is the $c-$transform of $\psi$ and $\psi=\phi^c$ is the $c-$transform of $\phi$). 
Hence, the dual problem can be expressed in terms of one only function $\phi$, taking $\psi=\phi^c$ (which automatically implies the constraint $\phi(x)+\phi^c(y)\leq c(x,y)$. Any optimal $\phi$ is called {\it Kantorovich potential}.
\item Given $\gamma$ and a pair $(\phi,\psi)$ then $\gamma$ is optimal for (K) and $(\phi,\psi)$ is optimal for (D) if and only if the equality $\phi(x) + \psi(y)= c(x,y)$ holds for all $(x,y)\in\spt\gamma$.
\end{itemize}
An interesting consequence of this last property is the fact that the support $\Gamma$ of any optimal $\gamma$ is $c-$cyclically monotone.

\begin{defi} Given a function $c:\R^d\times\R^d\to\R$, we say that a set $\Gamma\subset\R^d\times\R^d$ is $c-$cyclically monotone (briefly $c-$CM) if, 
for every $k\in\N$, every permutation $\sigma$ of $k$ elements and every finite family of points $(x_1,y_1),\dots,(x_k,y_k)\in \Gamma$ we have
$$\sum_{i=1}^k c(x_i,y_i)\leq \sum_{i=1}^k c(x_i,y_{\sigma(i)}).$$
The word ``cyclical'' refers to the fact that we can restrict our attention to cyclical permutations. The word ``monotone'' is a left-over from the case $c(x,y)=-x\cdot y$.
\end{defi}
Indeed, it is easy to check that $\Gamma$ is  $c-$CM from the fact that $\sum_{i=1}^k c(x_i,y_i)=\sum_{i=1}^k \phi(x_i)+\psi(y_i)=\sum_{i=1}^k \phi(x_i)+\psi(y_{\sigma(i)})\leq \sum_{i=1}^k c(x_i,y_{\sigma(i)}).$ 
This fact has a very interesting consequence in the case we consider in this paper, i.e. $c(x,y)=\ell(|x-y|)$ with $\ell$ increasing and strictly concave. As in \cite{GaMc}, we start from:

\begin{lemma}\label{strict dist}
Let $\ell:\R_+\to \R_+$ be a strictly concave and increasing function with $\ell(0)=0$. Then $\ell$ is strictly subadditive, i.e. $\ell(s+t)<\ell(s)+\ell(t)$ for every $s,t>0$. 
Also, if $x,y,z$ are points in $\R^n$ with $x\neq y$ and $y\neq z$, then $\ell(|x-z|)<\ell(|x-y|)+\ell(|y-z|)$. 
\end{lemma}
\begin{proof}
The subadditivity of positive concave functions is a classical fact which can be proven in the following way. Take $t,s>0$ and consider the function $g:[0,t+s]\to\R_+$ defined through $g(r)=\ell(t+s-r)+\ell(r)$. 
Then $g$ is strictly concave (as a sum of two strictly concave functions), and hence its minimal value is attained (only) on the boundary of the interval. 
Since $g(0)=g(t+s)=\ell(t+s)+\ell(0)=\ell(t+s)$ we get $\ell(s)+\ell(t)=g(t)>\min g = g(0)=\ell(t+s)$.

The second part of the statement is an easy consequence of the triangle inequality and the monotonicity of $\ell$:
$$\ell(|x-z|)\leq \ell(|x-y|+|y-z|)<\ell(|x-y|)+\ell(|y-z|).\qedhere$$
\end{proof}

This can be applied to the study of optimal transport plans as in Proposition 2.9 of \cite{GaMc}.
\begin{thm}\label{stay at rest}
Let $\gamma$ be an optimal transport plan for the cost $c(x,y)=\ell(|x-y|)$ with $\ell:\R_+\to \R_+$ strictly concave, increasing, and such that $\ell(0)=0$. 
Let $\gamma=\gamma_D+\gamma_O$, where $\gamma_D$ is the restriction of $\gamma$ to the diagonal $D=\{(x,x)\,:\,x\in\R^d\}$ and $\gamma_O$ is the part outside the diagonal, i.e. the restriction to $D^c=(\R^d\times\R^d)\setminus D$. 
Then this decomposition is such that $(\pi_x)_\#\gamma_O$ and $(\pi_y)_\#\gamma_O$ are mutually singular measures.
\end{thm}
\begin{proof}
It is clear that $\gamma_O$ is concentrated on $\spt\gamma\setminus D$ and hence $(\pi_x)_\#\gamma_O$ is concentrated on $\pi_x(\spt\gamma\setminus D)$ and $(\pi_y)_\#\gamma_O$ is concentrated on $\pi_y(\spt\gamma\setminus D)$. 
We claim that these two sets are disjoint. Indeed suppose that a common point $z$ belongs to both. 
Then, by definition, there exists $y$ such that $(z,y)\in \spt\gamma\setminus D$ and $x$ such that $(x,z)\in \spt\gamma\setminus D$. 
This means that we can apply $c-$cyclical monotonicity to the points $(x,z)$ and $(z,y)$ and get
$$\ell(|x-z|)+\ell(|z-y|)\leq \ell(|x-y|)+\ell(|z-z|)=\ell(|x-y|)<\ell(|x-z|)+\ell(|z-y|),$$
where the last strict inequality, justified by Lemma \ref{strict dist}, gives a contradiction.
\end{proof}

As we said in the introduction, the theorem above has some important consequences. In particular, it states that for this class of ``concave'' costs the common mass between $\mu$ and $\nu$ must stay at rest, which we can explain in details. 
Indeed, we have $(\pi_x)_\#\gamma_D=(\pi_y)_\#\gamma_D$ and since the remaining parts $(\pi_x)_\#\gamma_O$ and $(\pi_y)_\#\gamma_O$ are mutually singular, the decompositions 
$\mu=(\pi_x)_\#\gamma_D+(\pi_x)_\#\gamma_O$ and $\nu=(\pi_y)_\#\gamma_D+(\pi_y)_\#\gamma_O$ imply $(\pi_x)_\#\gamma_D=(\pi_y)_\#\gamma_D=\mu\wedge\nu$, $(\pi_x)_\#\gamma_O=(\mu-\nu)_+$ and $(\pi_y)_\#\gamma_O=(\nu-\mu)_+$. 
Hence, we must look in particular at the optimal transport problem between the two mutually singular measures $(\mu-\nu)_+$ and $(\mu-\nu)_-$. 
We will show that, under some natural regularity assumptions on the starting measure, this problem admits the existence of an optimal map. 

From now on, we will just assume w.l.o.g. that $\mu$ and $\nu$ are mutually singular. If not, just remove the common part, since we can deal with it separately.

Let us see which is the general strategy for proving existence of an optimal $T$ when the cost $c$ is of the form $c(x,y)=h(x-y)$ (not necessarily depending only on the norm $|x-y|$). Here as well the duality plays an important role.  
Indeed, if we consider an optimal transport plan $\gamma$ and a Kantorovich potential $\phi$ with $\psi=\phi^c$, we may write
$$\phi(x)+\psi(y)\leq c(x,y) \mbox{ on }\R^d\times\R^d \mbox{ and }\phi(x)+\psi(y)= c(x,y) \mbox{ on }\spt\gamma.$$

Once we have that, let us fix a point $(x_0,y_0)\in\spt\gamma$. One may deduce from the previous computations that
\begin{equation}\label{minsupp}
x\mapsto \phi(x)-h(x-y_0)\quad\mbox{ is maximal at }x=x_0
\end{equation}
and, if $\phi$ is differentiable at $x_0$ and $x_0$ is a interior point, one gets $\nabla\phi(x_0)=\nabla h(x_0-y_0)$ (we also assume that $h$ is differentiable at $x_0-y_0$).
The easiest case is that of a strictly convex function $h$, since one may inverse the relation passing to $\nabla h^*$ thus getting
$$x_0-y_0=\nabla h^*(\nabla \phi(x_0)).$$
This allows to express $y_0$ as a function of $x_0$, thus proving that there is only one point $(x_0,y)\in\spt\gamma$ and hence that $\gamma$ comes from a transport $T(x)=x-\nabla h^*(\nabla \phi(x))$. 
It is possible to see that strict convexity of $h$ is not the important assumption, but we need $\nabla h$ to be injective, which is also the case if $h(z)=\ell(|z|)$ since $\nabla h(z)=\ell'(|z|)\frac{z}{|z|}$ and the modulus of this vector identifies the modulus $|z|$ 
(since $\ell'$ is strictly increasing) and the direction gives also the direction of $z$. 
We will deal later with the case where $\ell$ is not differentiable at some points, and it will not be so difficult. 
The only difficult point is $0$ since $h$ would be highly singular at the origin, but fortunately if $\mu$ and $\nu$ have no mass in common then the case $x_0=y_0$ will be negligible.

However, we need to guarantee that $\phi$ is differentiable a.e. with respect to $\mu$. This is usually guaranteed by requiring $\mu$ to be absolutely continuous with respect to the Lebesgue measure, 
and using the fact that $\phi(x)=\inf_y h(|x-y|)-\psi(y)$ allows to prove Lipschitz continuity of $\phi$ if $h$ is Lipschitz continuous. 
This is indeed the main difficulty: concave functions on $\R^+$ may have an infinite slope at $0$ and be non-Lipschitz. 
The fact that the case of distance $0$ is negligible is not enough to restrict $h$ to a set where it is Lipschitz, if we do not have a lower bound on $|x-y|$. 
This can be easily obtained if one supposes that $\spt\mu\cap\spt\nu=\emptyset$, but this is in general not the case when $\mu$ and $\nu$ are obtained from the two original measures by removing the common mass. 
The paper by Gangbo and McCann makes a slightly less restrictive assumption, i.e. that $\mu(\spt\nu)=0$. 
This allows to say that $\mu-$almost any point is far from the support of $\nu$ and, since we only need local Lipschitz continuity, this is enough.

In the next two sections we will develop two increasingly hard arguments to provide $\mu-$a.e. differentiability for $\phi$. The next session will assume $\mu$ to be absolutely continuous and will make use of the notion of approximate gradient. 
Then, we will weaken the assumptions on $\mu$ supposing only that it does not give mass to $(d\!-\!1)-$rectifiable sets. 
This will require a different ad-hoc notion of gradient and a geometric measure theory lemma which replaces the notion of Lebesgue points. 
The first proof in the case $\mu\ll\lcal^d$ is given for the sake of simplicity and completeness. 
A consequence of this all is the existence of a transport map for $c(x,y)=\ell(|x-y|)$ every time that 
\begin{enumerate}[a)]
\item $\mu$ and $\nu$ are mutually singular 
\item $\mu$ does not give mass to ``small sets''. 
\end{enumerate}
This map can also be proven to be unique, as it is the case every time that we can prove that every optimal plan is indeed induced by a map. 
Moreover, if we only suppose b) and we admit the existence of a common mass between $\mu$ and $\nu$ (actually it is even enough to suppose that $\mu-\mu\wedge\nu$ does not give mass to ``small sets''), then the optimal $\gamma$ is unique and composed of two parts: 
one on the diagonal and one induced by a transport map, which implies that every point has at most two images, one of the two being the point itself.

We finish this section with a - standard - counterexample (see figure below) which shows that the assumption that  $\mu$ does not give mass to ``small sets'' is natural. Indeed, one can consider 
$$\mu=\huno\res A\;\mbox{ and }\;\nu=\frac{\huno\res B+\huno\res C}{2}$$
where $A$, $B$ and $C$ are three vertical parallel segments in $\R^2$ whose vertexes lie on the two line $y=0$ and $y=1$ and the abscissas are $0$, $1$ and $-1$, respectively, and $\huno$ is the $1-$dimensional Haudorff measure. 
It is clear that no transport plan may realize a cost better than $1$ since, horizontally, every point needs to be displaced of a distance $1$. 
Moreover, one can get a sequence of maps $T_n: A\to B\cup C$ by dividing $A$ into $2n$ equal segments $(A_i)_{i=1,\dots,2n}$ and $B$ and $C$ into $n$ segments each, $(B_i)_{i=1,\dots,n}$ and $(C_i)_{i=1,\dots,n}$ (all ordered downwards). 
Then define $T_n$ as a piecewise affine map which sends $A_{2i-1}$ onto $B_i$ and $A_{2i}$ onto $C_i$. 
In this way the cost of the map $T_n$ is less than $\ell(1+1/n)$, which implies that the infimum of the Kantorovich problem is $\ell(1)$, as well as the infimum on transport maps only. 
Yet, no map $T$ may obtain a cost $\ell(1)$, as this would imply that all points are sent horizontally, but this cannot respect the push-forward constraint. 
On the other hand, the transport plan associated to $T_n$ weakly converge to the transport plan $\frac 12 T^+_{\#}\mu+\frac 12 T^-_{\#}\mu,$ where $T^{\pm}(x)=x\pm e$ and $e=(1,0)$. 
This transport plan turns out to be the only optimal transport plan and its cost is $\ell(1)$.

In this example we have two measures with disjoint supports and no optimal transport map (every point $x$ is forced to have two images $x\pm e$), but we can add a common mass, for instance taking 
$$\mu=\huno\res A\;\mbox{ and }\;\nu=\frac\mu 2+\frac{\huno\res B+\huno\res C}{4}$$
and in such a case every point should have three images (part of the mass at $x$ is sent at $x\pm e$ and part stays at $x$).

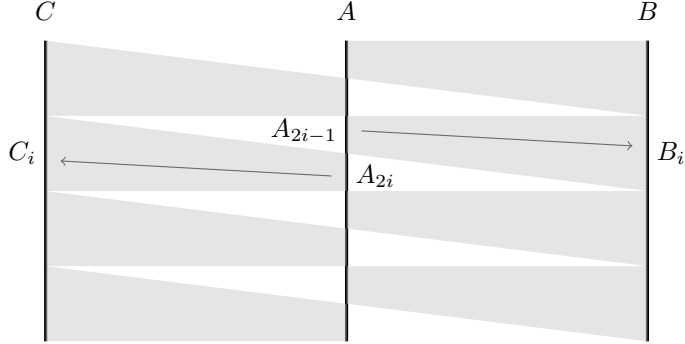
\begin{tikzpicture}

\draw (0,5.4) node{$A$};
\draw (4,5.4) node{$B$};
\draw (-4,5.4) node{$C$};

\draw (-4,3.5) node[left]{$C_i$};
\draw (4,3.5) node[right]{$B_i$};
\draw (0,3.2) node[right]{$A_{2i}$};
\draw (0,3.8) node[left]{$A_{2i-1}$};

\draw[very thick] (0,1)--(0,5);
\draw[very thick] (4,1)--(4,5);
\draw[very thick] (-4,1)--(-4,5);

\draw[->] (-0.2,3.2)--(-3.8,3.4);
\draw[->] (0.2,3.8)--(3.8,3.6);

\fill[gray!40, opacity=0.5] (0,4)--(-4,4)--(-4,5)--(0,4.5)--cycle;
\fill[gray!40, opacity=0.5] (0,4.5)--(0,5)--(4,5)--(4,4)--cycle;

\fill[gray!40, opacity=0.5] (0,3)--(-4,3)--(-4,4)--(0,3.5)--cycle;
\fill[gray!40, opacity=0.5]  (0,3.5)--(0,4)--(4,4)--(4,3)--cycle;

\fill[gray!40, opacity=0.5] (0,2)--(-4,2)--(-4,3)--(0,2.5)--cycle;
\fill[gray!40, opacity=0.5] (0,2.5)--(0,3)--(4,3)--(4,2)--cycle;

\fill[gray!40, opacity=0.5] (0,1)--(-4,1)--(-4,2)--(0,1.5)--cycle;
\fill[gray!40, opacity=0.5] (0,1.5)--(0,2)--(4,2)--(4,1)--cycle;

\end{tikzpicture}

\subsection{The approximate gradient}

In this section we recall some notions about a measure-theoretical notion replacing the gradient for less regular functions. The interested reader can find many details in \cite{EvaGar}. 

Let us start from the following observation: given a function $f:\Omega\to\R$ and a point $x_0\in\Omega$, we say that  $f$ is differentiable at $x_0\in\Omega$ and that its gradient is $\nabla f(x_0)\in\R^d$ if for every $\epsilon>0$ the set 
$$\{x\in\Omega\,:\,
\abs{f(x)-f(x_0)-\nabla f(x_0)\cdot (x-x_0)} > \epsilon \abs{x-x_0}\}$$
is at positive distance from $x_0$, i.e. if there exist a ball around $x_0$ which does not meet it. 
 Instead of this requirement, we could ask for a weaker condition, namely that $x_0$ is a zero-density point for the same set (i.e. a Lebesgue point of its complement). More precisely, if there exists a vector $v$ such that 

$$\lim_{\delta\to 0}
\frac{\abs{\{x\in\Omega\,:\,
\abs{f(x)-f(x_0)-v\cdot (x-x_0)} > \epsilon \abs{x-x_0}\}}}
{\abs{B(x_0 ,\delta)}}
=0$$
then we say that $f$ is approximately differentiable at $x_0$ and its approximate gradient is $v$.
The approximate gradient will be denoted by $\nabla_{app}f(x_0)$. As one can expect, it enjoys several of the properties of $v(x_0)$, that we list here.
\begin{itemize}
\item The approximate gradient, provided it exists, is unique.
\item The approximate gradient is nothing but the usual gradient if $f$ is differentiable.
\item The approximate gradient shares the usual algebraic properties of gradients, in particular $\nabla_{app}(f+g)(x_0)=\nabla_{app}f(x_0)+\nabla_{app}g(x_0)$.
\item If $x_0$ is a local minimum or local maximum for $f$, and if $\nabla_{app}f(x_0)$ exists, then $\nabla_{app}f(x_0)=0$.
\end{itemize}
these four properties are quite easy to check.

Another very important property that we need is a consequence of the well-know Rademacher theorem, which states that Lipschitz functions are almost everywhere differentiable.

\begin{prop}\label{loc_lip_app_grad}
Let $f,g:\Omega\to\R$ be two functions defined on a same domain $\Omega$ with $g$ Lipschitz continuous. Let $A\subset\Omega$ be a Borel set such that $f=g$ on $A$. 
Then $f$ is approximately differentiable almost everywhere on $A$ and $\nabla_{app}f(x)=\nabla g(x)$ for a.e. $x\in A$.
\end{prop}

\begin{proof}
It is enough to consider all the points in $A$ which are Lebesgue points of $A$ and at the same time differentiability points of $g$. These points cover almost all $A$. 
It is easy to check that the definition of approximate gradient of $f$ at a point $x_0$ is satisfied if we take $v=\nabla g(x_0)$.\end{proof}

\section{The absolutely continuous case}

Suppose now $\mu \ll \text{Leb}$, and proceed according to the strategy presented in Section 2.1.
  Take $x_0\in \spt(\mu)$: there exists a point $y_0\in\R^d$ such that $(x_0,y_0)\in\spt(\gamma)$. Denote by $\phi$ a Kantorovich potential and by $\phi^c$ its conjugate function.
From Equation \eqref{minsupp}, provided we can differentiate, we get
\begin{equation}\label{gradgrad}
0=\nabla \phi(x_0)-\nabla \ell(\abs{x_0-y_0})=\nabla\phi(x_0)-\ell'(\abs{x_0-y_0})\frac{x_0-y_0}{\abs{x_0-y_0}}\ ,
\end{equation}
 so that $x_0$ uniquely determines $y_0$, but unfortunately neither $\phi$ nor $\ell$ are smooth enough to differentiate. To comply with this, we first want to prove $\phi$ admits an 
approximate gradient Lebesgue-a.e. 
From what we saw in Section 2.2 this would imply (if $\ell$ is differentiable, we will see later how to handle the case where it is not) that Equation \eqref{gradgrad} is satisfied if we replace the gradient with the approximate gradient.

 Recall that we may suppose
\[
\phi(x)=\phi^{cc}(x)=\inf_{y\in\R^d}\ell(\abs{x-y})-\phi^c(y)\ .
\]

 Now 
%call $\R^d$ any -- until now this is enough -- smooth set containing both $\spt\mu$ and $\spt\nu$ and 
consider a countable family of closed balls $B_{i}$ generating the topology of $\R^d$, and
for every $i$ consider the function defined as 
\[
\phi_{i}(x):=\inf_{y\in B_{i}}\ell(\abs{x-y})-\phi^c (y)\ .
\]
for $x\in \R^n$. One cannot provide straight Lipschitz properties for $\phi_i$, since a priori $y$ is arbitrarily
close to $x$ and in general $\ell$ is not Lipschitz close to $0$. However $\phi_i$ is Lipschitz on every $B_j$ such that $\text{dist}(B_i,B_j)>0$. Indeed if $x\in B_j$,
$y\in B_i$ one has $\abs{x-y}\geq d>0$, therefore the Lipschitz constant of $\ell(\abs{\cdot-y})-\phi^c(y)$ does not exceed
$\ell'(d)$ (or the right derivative of $\ell$ at $d$). It follows that $\phi_i$ is Lipschitz on $B_j$, and its constant does not exceed
$\ell'(d)$.

Then, by Proposition \ref{loc_lip_app_grad}, $\phi$ has an approximate gradient almost everywhere
on $\{\phi=\phi_i\}\cap B_{j}$. By countable union, $\phi$ admits
an approximate gradient a.e. on
\[
\bigcup_{\substack{i,j \\ d(B_i,B_j)>0}}[\{\phi_i =\phi\}\cap B_j]\ .
\]
 As a consequence of this and of the absolute continuity of $\mu$, in order to prove that $\phi$ has an approximate gradient $\mu$-almost everywhere,
it is enough to prove that

\[
\mu \left(\array{@{}c@{}}
\displaystyle\bigcup\\
  \substack{i,j \\ d(B_i,B_j) > 0}\endarray 
\raisebox{+1ex}{$\{\phi_i = \phi \} \cap B_j$}\right) = 1.
\]

%\[
%\mu\Big(\bigcup_{\substack{i,j \\ d(B_i,B_j)>0}}\{\phi_i =\phi\}\cap B_j\Big)=1\ .
%\]

% Since $\ell$-transforms share the modulus of continuity of $\ell$, $\phi_{i}$ are all continuous. It follows that if
%$x\in\{\phi=\phi_i\}\cap B_j$
%\[
%\phi(x)=\inf_{y\in B_i}l(\abs{x-y})-\phi^c(y)=\min_{y\in B_i}l(\abs{x-y})-\phi^c(y)
%\]
%and therefore there exists a $y\in B_i$ such that $\phi(x)=l(\abs{x-y})-\phi^c(y)$. In other words
%\[
%
%\]

\begin{lemma}\label{phi_i}

For every $i$ and $j$ 
\[
\pi_x(\spt\gamma\cap (B_j\times B_i))\subset \{\phi=\phi_i\}\cap B_j\ . 
\]
\end{lemma}

\begin{proof}
Let $(x,y)\in\spt\gamma\cap (B_j\times B_i)$. Then $\phi(x)+\phi^c(y)=l(\abs{x-y})$. It follows that
\[
\phi_i(x)=\inf_{y'\in B_i}\ell(\abs{x-y'})-\phi^c(y')\leq\ell(\abs{x-y})-\phi^c(y)=\phi(x)\ . 
\]
 On the other hand, for every $x\in\R^n$
\[
\phi_i(x)=\inf_{y\in B_i}\ell(\abs{x-y})-\phi^c (y)\geq \inf_{y\in \R^n}\ell(\abs{x-y})-\phi^c(y)=\phi(x)\ .\qedhere
\] 
\end{proof}

 As a consequence of this,

\begin{align*}
\mu \left(\array{@{}c@{}}
\displaystyle\bigcup\\
  \substack{i,j \\ d(B_i,B_j) > 0}\endarray 
\raisebox{+1ex}{$\{\phi_i = \phi \} \cap B_j$}\right)
&\geq \mu \left(\array{@{}c@{}}
  \displaystyle\bigcup\\
  \substack{i,j \\ d(B_i,B_j) > 0}\endarray 
\raisebox{+1ex}{$\pi^x(\spt \gamma \cap (B_j \times B_i))$}\right)\\
&= \mu \left( \pi^x \left( \spt \gamma \cap \array{@{}c@{}}
  \displaystyle\bigcup\\
  \substack{i,j \\ d(B_i,B_j) > 0}\endarray 
\raisebox{+1ex}{$B_j \times B_i$} \right)\right)\\
&= \mu(\pi^x(\spt \gamma \setminus D))\\
&= \gamma \left[ (\pi^x)^{-1}(\pi^x(\spt \gamma \setminus D))\right]\\
&\geq \gamma(\spt \gamma \setminus D) = 1\\ 
\end{align*}
%
%\begin{multline*}
%\mu\Big(\bigcup_{\substack{i,j \\ d(B_i,B_j)>0}}\{\phi_i =\phi\}\cap B_j\Big)\geq 
%\mu\left(\bigcup_{\substack{i,j \\ d(B_i,B_j)>0}}\pi^x(\spt\gamma\cap (B_j\times B_i))\right) \\
%=\mu\left(\pi^{x}\left(\spt\gamma\cap\bigcup_{\substack{i,j \\ d(B_i,B_j)>0}}(B_j\times B_i)\right)\right) \\
%=\mu (\pi^{x}(\spt\gamma \setminus D)) \\
%=\gamma [(\pi^{x})^{-1} (\pi^{x}(\spt\gamma \setminus D))] \\
%\geq \gamma (\spt\gamma\setminus D)=1
%\end{multline*}
%
 since the diagonal is $\gamma$-negligible. In other words the
following theorem is proved

\begin{thm}
Let $\ell:[0,+\infty)\to\R_+$ be a concave function, and suppose $\mu$ and $\nu$ are two probability measures on $\R^n$, with
$\mu\ll\text{Leb}$. Call $\phi$ a Kantorovitch potential for the transport problem with cost $\ell(\abs{x-y})$. Then 
$\phi$ admits an approximate gradient $\mu$-a.e.
\end{thm}

We now come back to the proof of  the main theorem of the paper, under the additional assumption that $\mu$ is absolutely continuous.
\begin{thm}\label{simplified}
Suppose that $\mu$ and $\nu$ are two mutually singular probability measures on $\R^d$ such that $\mu\ll\lcal^d$, and take the cost $c(x,y)=\ell(|x-y|)$, for $\ell:\R_+\to \R_+$ strictly concave and increasing. 
Then there exists a unique optimal transport plan $\gamma$ and it is induced by a transport map.
\end{thm}
\begin{proof}
We will prove that any optimal transport plan $\gamma$ is induced by a transport map, which also implies uniqueness by standard techniques (see \cite{villani}). 
From the strategy that has been presented in Section 2.1, and since we know that $\gamma$ is concentrated outside the diagonal $D$, it is enough to prove that, if $(x_0,y_0)\in\spt\gamma\setminus D$, then $y_0$ is uniquely determined by $x_0$

\begin{itemize}

\item Case 1: $\ell$ differentiable at $\abs{x_0-y_0}$.

\noindent Let $(x_0, y_0)\in\spt\gamma\setminus D$, and suppose that $\ell$ is differentiable at $\abs{x_0-y_0}$. Then, if $\phi$ is approximately 
differentiable at $x_0$ (which is true $\mu-$a.e.)
\[
0=\nabla_{app}\phi (x_{0})-\ell'(\abs{x_0-y_0})\frac{x_0-y_0}{\abs{x_0-y_0}}\ .
\]
 Thus, since $\ell$ was supposed strictly increasing and strictly concave,
\[
\abs{x_0-y_0}=(\ell')^{-1}(\nabla_{app}\phi (x_0)\ ,
\]
and
\begin{equation}\label{divise by 0}
\frac{y_0-x_0}{\abs{y_0-x_0}}=-\frac{\nabla_{app}\phi(x_0)}{\ell'(\abs{x_0-y_0})}\ .
\end{equation}
 In other words $y_0$ is uniquely determined by $x_0$.

\item Case 2: $\ell$ not differentiable at $\abs{x_0-y_0}$.

\noindent This second case is even more striking. Consider the values of the right and left derivatives $\ell'_r(\abs{x_0-y_0})$ and $\ell'_l(\abs{x_0-y_0})$ and pick a value $p\in I:= [\ell'_r(\abs{x_0-y_0}),\ell'_l(\abs{x_0-y_0}) ]$. 
We know by concavity that we have 
$$l(t)-l(\abs{x_0-y_0})\leq p(t-\abs{x_0-y_0})\quad\mbox{ for all $t\in\R$}.$$
Consider the function
\[
\phi(\cdot)-[l(\abs{x_0-y_0})+p(\abs{\cdot -y_0}-\abs{x_0-y_0})]\,
\]
defined on $\R^n$. Then, for every $x\in\R^n$,
\begin{align*}
\phi(x)&-[l(\abs{x_0-y_0})+p(\abs{x -y_0}-\abs{x_0-y_0})] \\
&\leq \phi(x)-l(\abs{x-y_0})
\leq \phi(x_0)-l(\abs{x_0-y_0})\\
&=\phi(x_0)-[l(\abs{x_0-y_0})+p(\abs{x_0 -y_0}-\abs{x_0-y_0})]\ .
\end{align*}
In other words, $\phi(\cdot)-[l(\abs{x_0-y_0})+p(\abs{\cdot -y_0}-\abs{x_0-y_0})]$ has a maximum in $x_0$. Since
$x_0\neq y_0$,
\[
0=\nabla_{app}\phi(x_0)-p\frac{x_0-y_0}{\abs{x_0-y_0}}\ .
\]
 It follows that $\abs{\nabla_{app}\phi(x_0)}=p$ for every $p\in I$, and this leads to a
contradiction.

\noindent This means that this second case can only occur on a negligible set of points $x_0$ (those where $\phi$ is not approximately differentiable).\qedhere
\end{itemize}
\end{proof}

\begin{rem}
The present paper is fully written under the assumption that $\ell$ is strictly concave {\it and increasing}, which implies $\ell'>0$. Yet, most of the analysis could be adapted to the case where $\ell$ is only supposed to be strictly concave, even if monotonicity is indeed very much reasonable from the modelization point of view. One of the problem in case $\ell'$ is not strictly positive is the fact that we risk to divide by $0$ in equation \eqref{divise by 0}. This can be fixed since it only happens at a maximum point for $\ell$, and $c-$cyclical monotonicity can be used to prove that $\gamma(\{(x,y)\,:\,|x-y|=m\}=0$ when $m=\argmax\ell$ together with a density points argument. Another difficulty to be fixed is the fact that with $\ell$ not increasing the cost $c$ is not anymore bounded from below, and we would need to assume $\mu$ and $\nu$ compactly supported. Yet, this is not in the scopes of the present paper.
\end{rem}

\section{The general case: $\mu$ does not give mass to small sets}\label{sec:general}

 Let us now consider weaker assumptions, i.e. suppose that $\mu$ does not give mass to
small sets. Namely suppose that, for every $A\subset\R^d$ which is $\mathcal{H}^{d\!-\!1}$-rectifiable, we have $\mu(A)=0$.

\subsection{A GMT lemma for density points}

The following lemma is an interesting result from Geometric Measure Theory that can be used instead of Lebesgue points-type results when we face a measure which is not absolutely continuous but ``does not give mass to small sets''. 
It states that, in such a case, $\mu-$a.e. point $x$ is such that every cone exiting from $x$, even if very small, has positive $\mu$ mass. 
In particular it means that we can find points of $\spt\mu$ in almost arbitrary directions close to $x$.

We give the proof of this lemma for the sake of completeness, and because it is not easy to find hints or references for it in the literature about GMT. 
Indeed, the result is known in a part of the optimal transport community, where density points (in particular Lebesgue points) play an important role in some strategies for the existence of optimal maps 
(see \cite{ChaDeP09,ChaDePJuu}), and it has recently been detailed in \cite{ChaDeP DCDS}. 
It is also part of the folklore of some branches of the GMT community, but, also in this case, it is not easy to find written evidence of it in this form and under these assumptions. 
Indeed, Lemma 3.3.5 in \cite{Fed} proves the $(d\!-\!1)-$rectifiability of any set such that every point admits the existence of a small two-sided cone containing no other point of the set, 
and this would allow to prove the statement we need, up to the fact that in our case we only  consider one-sided cones. 
The proof follows anyway the same main ideas, and the one that we propose here is especially written in terms of $\mu-$negligibility.

\begin{lemma}\label{Paul's}
Let $\mu$ be a Borel measure on $\R^{d}$, and suppose that $\mu$ does not charge small sets. Then $\mu$ is 
concentrated on the set
\[
\{x:\forall\epsilon >0,\forall\delta >0, \forall u\in \mathbb{S}^{d-1},
\ \mu (C(x,u,\delta,\epsilon))>0\}\ ,
\]
 where
\[
C(x,u,\delta,\epsilon)=C(x,u,\delta)\cap\overline{B(x,\epsilon)}:=
\{y:\langle y-x,u\rangle \geq (1-\delta)\abs{y-x}\}\cap\overline{B(x,\epsilon)}
\]
\end{lemma}

\begin{proof}
Equivalently we will prove that
\[
\mu\left(\bigcup_{u,\delta,\epsilon}\{x:\mu(C(x,u,\delta,\epsilon))=0\}\right)=0\ .
\]
First notice  that $u$, $\delta$ and $\epsilon$ may be taken each in a countable set. This means that it is enough to fix then \textit{una tantum} $u$, $\delta$ and $\epsilon$ and then prove that 
 \[
\mu\left(\{x:\mu(C(x,u,\delta,\epsilon))=0\}\right)=0\ .
\]
 Moreover, for sake of simplicity, suppose 
$u=(0,\ldots , -1)$. Take now all cubes $Q$ with sides parallel to the coordinate hyperplanes, centered in a point 
of $\Q^{d}$ and with sidelength belonging to $\Q_+$, and call the set of such cubes $\{Q_n\}$.
We can see that 
\[
\{x:\mu(C(x,u,\delta,\epsilon))=0\}\subset\bigcup_n \{x\in Q_n:\mu(Q_n\cap C(x,u,\delta))=0\}\ .
\]

 Therefore we will show, for a fixed cube $Q$, that
\[
\mu(\{x\in Q:\mu(Q\cap C(x,u,\delta))=0\})=0\ .
\]
 Now write every $y\in \R^d$ as $y=(y',y_d)$, with $y'\in\R^{d-1}$. Then a quick computation shows that
$(y',y_d)\in C((x',x_d),u,\delta)$ if and only if
\[
y_d\leq x_d -\frac{1-\delta}{\sqrt{\delta(2-\delta)}}\abs{y'-x'}\ ,
\]
and we set $k(\delta):=\frac{1-\delta}{\sqrt{\delta(2-\delta)}}$.
%
%i.e.
%\[
%C(x,u,\delta)=\hyp [{x_d -k(\delta)\abs{\cdot-x'}}]\ ,
%\]
% where $\hyp{f}:=\{(x,y)\in \R^{d-1}\times\R:x\in A,\ y\leq f(x)\}$, for $f:A\subset \R^{d-1}\to\R$, and
%$k(\delta):=\frac{1-\delta}{\sqrt{\delta(2-\delta)}}$.
 
 Define now $X$ as the projection of $Q$ along the first $d-1$ coordinates, and for every $x'\in X$
\[
z(x'):=\sup\{z\in\R:\mu (Q\cap C((x',z),u,\delta))=0\}\ .
\]

 Notice that the set in the supremum is never empty, by taking for instance $z\leq\min\{x_d:x\in Q\}$. 

Let us study the function $x'\mapsto z(x')$ and notice that it is $k(\delta)$-Lipschitz continuous: 
indeed, if we had $z(x'_1)<z(x'_0)-k(\delta)|x'_1-x'_0|$, then, the cone $C((x'_1,z(x'_1)),u,\delta)$ would be included in the interior of $C((x'_0,z(x'_0)),u,\delta)$ and hence the same would be true for $C((x'_1,z(x'_1)+t),u,\delta)$ for small $t>0$.
Yet, this implies $\mu(C((x'_1,z(x'_1)+t),u,\delta))=0$ and hence the supremum defining $z(x'_1)$ should be at least $z(x'_1)+t$, 
which is obviously a contradiction. By interchanging the role of $x'_1$ and $x'_0$ we get $|z(x'_1)-z(x'_0)|\leq k(\delta)|x'_1-x'_0|$.

Now we take
\begin{multline*}
\{x\in Q:\mu(Q\cap C(x,\delta,u))=0\}=\{(x',x_d)\,:\,x_d\leq z(x')\}\\=\{(x',x_d)\,:\,x_d< z(x')\}\cup\{(x',x_d)\,:\,x_d= z(x')\}.
\end{multline*}
The second set on the right hand side is $\mu-$negligible since it is the graph of a Lipschitz function of $(d\!-\!1)$ variables and $\mu$ is supposed not to give mass to these sets. 
The first set is also negligible since it is contained in the complement of $\spt\mu$. 
Indeed, if a point $x=(x',x_d)$ satisfies $x_d< z(x')$, then we also have $x_d+t<z(x')$ for small $t>0$ and the interior of the cone  $C((x',z(x')+t),u,\delta)$ is $\mu-$negligible. 
This means that $x$ has a neighborhood where there is no mass of $\mu$, and hence that it does not belong to $\spt\mu$.

This ends the proof.
\end{proof}

Notice that the above result is obviously false if we withdraw the hypothesis on $\mu$: take a measure concentrated on a $(d\!-\!1)-$manifold (for instance, an hyperplane).
Then, for small $\ve$ and $\delta$, the measure $\mu(C(x,u,\delta,\epsilon))$ will be zero if $u$ does not belong to the tangent space to the manifold at $x$.

\subsection{How to handle non-absolutely continuous measures $\mu$}

 Consider again the transport problem, and suppose $\mu$ does not charge small sets.
Call $\gamma$ the minimizer of the Kantorovitch problem, $\phi$ a maximizer of the dual problem. As in 
Section 3, we take a countable family of balls $B_i$ generating the 
topology of $\R^n$, and for each $i$ define
\[
\phi_i (x):=\inf_{y\in B_i}l(\abs{x-y})-\phi^c(y),\qquad x\in\R^d\ .
\]

\begin{lemma}
If $B\subset\R^d$ is such that $d(0,B)=d_0>0$ and $\ell$ is concave and increasing, then the function $B\ni x\mapsto \ell(|x|)$ is semi-concave, and, more precisely, $x\mapsto \ell(|x|)-\frac 12 \ell'(d_0)|x|^2$ is concave (where $\ell'$ denotes the derivative, or right derivative in case $\ell$ is not differentiable at $d_0$).
\end{lemma}
\begin{proof}
We just need to give an upper bound on the second derivatives of $g(x):=\ell(|x|)$, which we will do in the case where $\ell$ is $C^2$. The general result will follow by approximation. Compute
$$\partial_i g(x)=\ell'(|x|)\frac{x_i}{|x|}\;;\quad \partial_{ij} g(x)=\ell''(|x|)\frac{x_ix_j}{|x|^2}+\ell'(|x|)\frac{|x|^2 \delta_{ij}-x_ix_j}{|x|^3}.$$
Hence, if we denote by $\hat x$ the unit vector $x/|x|$, the Hessian of $g$ is composed of two parts: 
the positive matrix $\hat x\otimes \hat x$ times the non-positive factor $\ell''(|x|)$, 
and the positive matrix $(\mathrm{Id}-\hat x\otimes \hat x)$ times the positive factor $\ell'(|x|)$. 
Since  $\mathrm{Id}-\hat x\otimes \hat x\leq \mathrm{Id}$ (in the sense of positive-definite symmetric matrices), this gives 
$$D^2g\leq \ell'(|x|)\mathrm{Id}\leq \ell'(d_0)\mathrm{Id}$$
and thesis follows.
\end{proof}

\begin{cor}\label{phi_i_semiconcave}
For every $i$ and $j$ such that $d(B_i,B_j)>0$, the function $\phi_i$ is semi-concave and hence $\mu$-a.e. differentiable on $B_j$. \end{cor}
\begin{proof}
The semiconcavity of $\phi_i$ follows from its definition as an infimum of semiconcave functions, with the same semicontinuity constant which is $\ell'(d(B_i,B_j))$.

This implies that $\phi_i $ has the same regularity and differentiability points of convex functions. It is well-known that convex functions are differentiable everywhere but on a $(d\!-\!1)-$rectifiable set: 
this is a consequence of the more general fact that the set where a convex-valued monotone multifunction in $\R^d$ takes values of dimension at least $k$ can be covered by countably many $(d-k)-$dimensional graphs of Lipschitz functions 
(see \cite{AlbAmb}), applied to the subdifferential multifunction.

Since $\mu$ does not give mass to small sets, the set of non-differentia\-bility points of $\phi_i$ is also negligible.\end{proof}

\begin{rem}\label{isotropy}
 For each pair $i$, $i'$ consider the measure
\[ 
\mu_{i,i'}:=\mu\res\{\phi=\phi_{i}=\phi_{i'}\},
\]
that is, for every $A$ Borel, $\mu_{i,i'}(A)=\mu(A\cap\{x:\phi(x)=\phi_{i}(x)=\phi_{i'}(x)\})$. Since for every small set
$A$ one has $\mu_{i,i'}(A)\leq \mu(A)=0$, Lemma \ref{Paul's} yields the existence of a $\mu$-negligible set
$N_{i,i'}\subset \{\phi=\phi_{i}=\phi_{i'}\}$ such that for every $x\in\{\phi=\phi_{i}=\phi_{i'}\}\setminus N_{i,i'}$
and for every $u\in S^{d-1}$, $\delta >0$, $\epsilon>0$, one has $\mu_{i,i'}(C(x,u,\delta,\epsilon))>0$. In particular
it follows that $C(x,u,\delta,\epsilon)\cap\{\phi=\phi_{i}=\phi_{i'}\}\neq \emptyset$.
\end{rem}

 This isotropy in the structure of $\{\phi=\phi_{i}=\phi_{i'}\}$ implies the following

\begin{prop}\label{gradientbar}
Consider
\[
N:=\bigcup_{i,i'}N_{i,i'}\cup
\bigcup_{\substack{i,j\\d(B_i,B_j)>0}}\{x\in B_j :\phi_{i}\text{ not differentiable at }x\}\ ,
\]
then $N$ is $\mu-$negligible. Moreover, consider $A:=(\pi_x)(\spt\gamma\setminus D)$, which is a set with $\mu(A)=1$: then, for every $x\in A\setminus N$ there exists $i$ such that $\phi(x)=\phi_{i}(x)$.

Moreover, if $\phi(x)=\phi_{i}(x)=\phi_{i'}(x)$, then
\[
\nabla \phi_{i}(x)=\nabla \phi_{i'}(x)\ .
\]
\end{prop}

\begin{proof}
The fact that $N$ is negligible follows from Remark \ref{isotropy} and Corollary \ref{phi_i_semiconcave}.

Let $x\in A\setminus N$. There exists $y$ such that $(x,y)\in \spt\gamma$ and $x\neq y$. It follows that there 
exist two balls $B_i$ and $B_j$ with $x\in B_j$, $y\in B_i$ and $d(B_i,B_j)>0$. Then, by Lemma \ref{phi_i}, 
$\phi(x)=\phi_{i}(x)$.

 Suppose now, for the sake of simplicity, $\phi(x)=\phi_{1}(x)=\phi_{2}(x)$. Then, since $x\notin N$, 
$v_1 :=\nabla\phi_{1}(x)$ and $v_2:=\nabla\phi_{2}(x)$ are both well-defined. By contradiction suppose $v_1\neq v_2$.
 Then for every $\epsilon>0$ there exists $\delta(\epsilon)>0$ such that for every $\tilde{x}\in B(x,\delta)$
\[
\abs{\phi_{1}(\tilde{x})-\phi(x)-v_1\cdot (\tilde{x}-x)}\leq\epsilon\abs{\tilde{x}-x}
\]
and
\[
\abs{\phi_{2}(\tilde{x})-\phi(x)-v_2\cdot (\tilde{x}-x)}\leq\epsilon\abs{\tilde{x}-x}\ .
\]
 Therefore, for such $\tilde{x}$,
\begin{multline*}
\abs{\phi_1(\tilde{x})-\phi_2(\tilde{x})+(v_2-v_1)\cdot(\tilde{x}-x)}\\=\abs{\phi_1 (\tilde{x})-\phi(x)-v_1\cdot(\tilde{x}-x)
-[\phi_2 (\tilde{x})-\phi(x)-v_2\cdot (\tilde{x}-x)]}\leq 2\epsilon\abs{\tilde{x}-x}
\end{multline*}

 In order to have a contradiction, it is enough to choose $\tilde{x}$ such that
\[
\phi_1(\tilde{x})=\phi_2(\tilde{x})
\] 
and
\[
(v_2 -v_1)\cdot (\tilde{x}-x)> 2\epsilon\abs{\tilde{x}-x}\ .
\]
 The latter may be expressed as
\[
\frac{(\tilde{x}-x)}{\abs{\tilde{x}-x}}\cdot\frac{v_2 -v_1}{\abs{v_2-v_1}}>\frac{2\epsilon}{\abs{v_2-v_1}}\ .
\]
In order to guarantee that this is possible, choose $\epsilon<\frac{\abs{v_2-v_1}}{2}$.
Then, by Remark \ref{isotropy} with 
$C\left(x,\frac{v_2-v_1}{\abs{v_2-v_1}},1-\frac{2\epsilon}{\abs{v_2-v_1}}, \delta(\epsilon)\right)$, we are done.
%$u=v_2-v_1$ and $\delta=1-\frac{2\epsilon}{\abs{v_2-v_1}}$, we are done.
\end{proof}

To go on with our proof we need the following lemma.

\begin{prop}\label{maximum}
Let $f:\Omega\to\R$ be differentiable at $x_0\in\Omega$. 
Suppose $B\subset\Omega$ is such that $x_0\in B$ and that for every $u\in S^{n-1}$, $\delta>0$ and $\epsilon>0$ one has
$B\cap C(x_0,u,\delta,\epsilon)\neq\emptyset$. Moreover suppose that $x_0$ is a local maximum for $f$ on $B$. Then 
\[
\nabla f(x_0)=0\ .
\]
\end{prop}

\begin{proof}
Write $v:=\nabla f(x_0)$ and suppose by contradiction $v\neq 0$. Then for every $\epsilon>0$ there exists 
$\delta(\epsilon)>0$ such that if $\abs{x-x_0}\leq\delta(\epsilon)$
\[
f(x)-f(x_0)\geq v\cdot (x-x_0)-\epsilon\abs{x-x_0}\ .
\]
 Moreover, from the assumption on $B$, we may suppose $x\in B\setminus\{x_0\}$,
\[
\frac{x-x_0}{\abs{x-x_0}}\cdot \frac{v}{\abs{v}}\geq \frac{1}{2}
\]
and
\[
\abs{x-x_0}\leq\delta(\epsilon)\ .
\]
 Then we have
\[
f(x)-f(x_0)\geq\abs{x-x_0} (\abs{v}/2-\epsilon)\ .
\]
Now choose $\epsilon<\frac{\abs{v}}{2}$. It follows that for every $\eta\leq \delta(\epsilon)$ we may choose 
$x\in B\setminus{x_0}$ such that $\abs{x-x_0}\leq\eta$ and
\[
f(x)>f(x_0)\ ,
\]
which is a contradiction.
\end{proof}

\begin{thm}\label{complete}
 Let $\mu$ and $\nu$ be Borel probability measures on $\R^d$, and suppose that $\mu$ does not charge small sets.
 Suppose moreover that $\mu$ and $\nu$ are mutually singular.
 Then, if $\ell :[0,+\infty )\to\R$ is strictly concave and increasing, there exists a unique optimal transport map for the cost $\ell(\abs{x-y})$.
\end{thm}
				\begin{proof}
We argue as in Section 3 and accordingly to the usual strategy. 

 Let us take $x_0\in A\setminus N$, with $(x_0,y_0)\in\spt\gamma$ and $x_0\neq y_0$. We just need to show that $y_0$ is uniquely determined by $x_0$. By Proposition \ref{gradientbar} 
$\phi(x_0)=\phi_{i}(x_0)$ for some $i$. Since
\[
x_0\in \argmax [\phi(\cdot)-\ell(\abs{\cdot-y_0})]\ ,
\]
in particular
\[
x_0\in \argmax [\phi_{i}(\cdot)-\ell(\abs{\cdot-y_0})]\restr_{\{\phi=\phi_{i}\}}\ .
\]

\begin{itemize}

\item Case 1: $\ell$ differentiable at $\abs{x_0-y_0}$.

It follows that $\phi_i(\cdot)-\ell(\abs{\cdot-y_0})$ is differentiable at $x_0$. From Proposition \ref{maximum} we have $\nabla\phi_i (x_0)-\nabla\ell(\abs{\cdot-y_0})\restr_{x_0}=0$. 
Notice that the value of the vector $\nabla\phi_i (x_0)$ only depends on $x_0$ and not on $y_0$, and does not depend on $i$.

By Remark \ref{isotropy} with $i=i'$ we may apply Proposition \ref{maximum} to yield that
\[
\nabla \phi_i(x_0)=\ell'(\abs{x_0-y_0})\frac{x_0-y_0}{\abs{x_0-y_0}}\ .
\]
It follows that
\[
\abs{y_0-x_0}=(\ell')^{-1}(\abs{\nabla \phi_i (x_0)})
\]
and
\[
\frac{x_0-y_0}{\abs{x_0-y_0}}=\frac{\nabla\phi_i (x_0)}{\ell'(\abs{x_0-y_0})}\ .
\]

\item Case 2: $\ell$ not differentiable at $\abs{x_0-y_0}$

Here as well we argue as in Section 3: pick a value $p\in I:=[\ell'_r(\abs{x_0-y_0}),\\
\ell'_l(\abs{x_0-y_0}) ]$. Suppose $\phi(x_0)=\phi_i(x_0)$ and consider the function
\[
\phi(\cdot)-[l(\abs{x_0-y_0})+p(\abs{\cdot -y_0}-\abs{x_0-y_0})]\ .
\]
defined on $\R^n$. Then, $\phi_{i}(\cdot)-[l(\abs{x_0-y_0})+p(\abs{\cdot -y_0}-\abs{x_0-y_0})]$ has a maximum point at $x_0$.  It follows that $\abs{\nabla\phi_{i}(x_0)}=p$ for every choice of $p$, and this leads to a
contradiction.
\end{itemize}
\end{proof}

\section{Appendix}
\subsection{Concave costs and translations}
 Here we wish to analyse a feature of strictly concave costs, namely that they penalize equal displacements.

\begin{prop}\label{transl}
Let $\mu$ and $\nu$ be mutually singular
Borel measures on $\R^d$, and suppose $\mu$ does not charge small sets. Let 
$l:[0,+\infty )\to\R$ be a $C^{2}$, increasing, strictly concave function, and suppose $\ell''(x)<0$ for every $x>0$.
Call $\gamma$ the optimal transport plan with respect to $\ell(\abs{\cdot})$. Then for every $e\in\R^d\setminus\{0\}$
\[
\gamma (\{(x,x+e):x\in\R^d\})=0\ .
\]
\end{prop}

\begin{proof}
For $e\in\R^d$ write $\gamma_{e}:=\gamma\res\{(x,x+e):x\in\R^d\}$. 
Suppose by contradiction that for some $e$ the measure $\gamma_e$ does not vanish. It follows that $\mu_e:=(\pi^{x})_{\#}\gamma_e$
is nontrivial as well. Moreover it does not charge small sets, since $\mu_e\leq\mu$. Therefore Lemma \ref{Paul's} implies
that $\mu_{e}$ is concentrated on
\begin{equation}
\{x:\forall\epsilon >0,\forall\delta >0, \forall u\in \R^d \text{ such that }\abs{u}=1,
\ \mu_{e} (C(x,u,\delta,\epsilon))>0\}\ ,
\end{equation}
which we will refer to as $A_{e}$. Clearly we have 
\[
A_{e}\subset\spt \mu_{e}=\pi_{x}(\spt \gamma _e)\ .
\]
 Now take $x_1$ and $x_2$ in $\spt\mu_{e}$. Then 
$(x_1,x_1+e), (x_2,x_2+e)\in\spt\gamma_{e}$, which is $l$-cyclically monotone. If we call $\xi:=x_2-x_1$
it follows that
\[
2 \ell(\abs{e})\leq \ell(\abs{e+\xi})+\ell(\abs{e-\xi})\ .
\]
 Write then a second order Taylor expansion, to yield
\[
2 \ell(\abs{e})\leq 2 \ell(\abs{e})+\sum_{i,j}\xi_{i}
\bigg[\ell''(\abs{e})\frac{e_i e_j}{\abs{e}^2}+\ell'(\abs{e})\bigg(\frac{\delta_{ij}}{\abs{e}}-\frac{e_i e_j}{\abs{e}^3}\bigg)
\bigg]\xi_{j}+o(\abs{\xi}^2)
\]
or, in other words,
\begin{align*}
0&\leq\sum_{i,j}\xi_{i}
\bigg[\ell''(\abs{e})\frac{e_i e_j}{\abs{e}}+\ell'(\abs{e})\bigg(\delta_{ij}-\frac{e_i e_j}{\abs{e}^2}\bigg)
\bigg]\xi_{j}+o(\abs{\xi}^2)\\
&=\frac{\ell''(\abs{e})}{\abs{e}}(e\cdot \xi)^{2}+\ell'(\abs{e})\abs{\xi}^{2}
-\frac{\ell'(\abs{e})}{\abs{e}^{2}}(e\cdot \xi)^{2}+o(\abs{\xi}^{2})\ .
\end{align*}
 Now take any $x_1 \in A_e$, which is nonempty by our hypothesis. Then there exists $x_2\in \spt\mu_{e}$ such that $\xi\cdot e \geq (1-\delta)\abs{\xi}\abs{e}$
with $\delta$ arbitrarily small. Then by plugging this into the previous equation, and letting $\xi\to 0$, we obtain a contradiction.
\end{proof}

The conclusion of the previous proposition is strongly different from the convex case. Indeed, consider any Borel measure $\mu$ on $\R^d$ with compact support, and fix a vector 
$e\in \R^d\setminus \{0\}$. Define $\tau _{e}$ to be the translation by $e$ and
\[
\mu_{e}:=(\tau_{e})_{\#}\mu
\]
as the translated $\mu$. Then one may prove that $\tau_{e}$ is an optimal transportation of $\mu$ onto $\mu_{e}$ if the cost is convex, while Proposition \ref{transl} implies that translations are never optimal for concave costs.

In order to prove that $\tau_{e}$ is an optimal transportation of $\mu$ onto $\mu_{e}$,
consider any map $T$ which pushes forward $\mu$ to $\mu_e$. Then, by Jensen inequality, if $x\mapsto \ell(|x|)$ is convex, we have
\begin{align*}
\int c(x,Tx)d\mu &=\int \ell(\abs{x-Tx})d\mu \geq \ell\bigg(\abs{\int x-Tx\ d\mu}\bigg)\\
&=\ell\bigg(\abs{\int x\ d\mu -\int x\ d\mu_{e}}\bigg)=\ell(\abs{e})\\
&=\int \ell(\abs{x-\tau_{e}(x)})\ d\mu\ .
\end{align*}

\subsection{On the definition of a weak $\mu-$approximate gradient }

Even if not strictly necessary for the sake of the paper, in this section we discuss the possibility of using Lemma \ref{Paul's} in order to define a sort of approximate gradient with respect to a measure $\mu$ which gives no mass to small sets. The idea is that in Section 4 we produced an ad-hoc choice of gradient, can we extend it to more general frameworks in order to use it in other situations? The answer will be both yes and no.

Before entering into details, let us set the following language.

\begin{defi} Let $\mu \in \calP(\R^n)$. We define the isotropic set of $\mu$ by
$$\isot \mu = \{ x : \forall \delta, u, \epsilon, \mu(C(x,u,\delta,\epsilon)) > 0 \}.$$

If $A$ is a Borel subset, the $\mu$-isotropic set $A$ is by definition
$$\isot_\mu (A) = A \cap \isot (\mu\res A).$$
\end{defi}

We know that, if $\mu$ does not charge small sets, then for all Boret subset $A$ (by Lemma \ref{Paul's}) we have
$$\mu(A \setminus \isot_\mu(A)) = 0.$$
We now pass to a naive definition of gradient.

\begin{defi}
We say that $L \in \R^n$ is a $\mu$-isotropic gradient of $f : A \to \R$ at $x$ if for all $\epsilon > 0$, $x$ belongs to the set
$$\isot_\mu \left(\left\{ y : \abs{f(y) - (f(x) + \langle L, y-x \rangle}) \leq \epsilon\abs{y-x}\right\}\right).$$
\end{defi}

This definition only means that we can find points $y \in \spt \mu\res D_\epsilon$ where
$$D_\epsilon = \abs{f(y) - (f(x) + \langle L, y-x \rangle}) \leq \epsilon \abs{y-x}$$
in almost arbitrary directions and arbitrary close to $x$. It satisfies some properties, for instance

\begin{prop}\label{criticpt}
If $L$ is a $\mu$-isotropic gradient of $f$ at $x$, where $f$ has a local extremum, then $L = 0$.
\end{prop}

We do not give the proof, which is quite similar to that of Proposition \ref{maximum} (indeed, most proofs will be skipped in this appendix subsection, since they simply recall the proofs of Section 4).
Yet, this definition is not enough to guarantee uniqueness. Indeed, take two disjoint sets $A,B\subset \mathbb{S}^1$ with $A\cup B=\mathbb{S}^1$, such that the supports of $\lcal^1\res A$ and $\lcal^1\res B$ are both the whole $\mathbb{S}^1$ (finding two such sets is a non-trivial, but classical exercise). Now take two different vectors $L_A$ and $L_B$ and define 
$$f(x)=\begin{cases}0&\mbox{if }x=0,\\
				L_A\cdot x&\mbox{if }\frac{x}{|x|}\in A,\\
				L_B\cdot x&\mbox{if }\frac{x}{|x|}\in B.\end{cases}$$
It follows that both $L_A$ and $L_B$ are $\mu$-isotropic gradient of $f$ at any $x \in \R^d$ such that $L_A\cdot x=L_B\cdot x$. Notice that this is a ``small set'': it leaves open the question of a possible $\mu-$a.e. uniqueness, but the situation is anyway worse than what happens for the approximate gradient (Section 2), where the gradient is necessarily unique at any point where it exists.

Hence, we try to switch to a better definition of gradient. We rely on the following observation, which is essentially contained in the proof of Proposition \ref{gradientbar}			
\begin{prop}\label{isot} 
 If $A \subseteq \Omega \subseteq \R^n$ where $A$ is a Borel subset, $\Omega$ open set, and $\phi, \psi : \Omega \to \R$, then $\nabla \phi (x) = \nabla \psi (x)$ on the set
$$\isot_\mu (\{ x : \phi(x) = \psi(x), \nabla \phi(x) \text{ and } \nabla \psi(x) \text{ exist}\}).$$
\end{prop}

The following proposition is an attempt at defining a uniquely determined $\mu$-approximate gradient under hypotheses satisfied by the Kantorovitch potential in Section \ref{sec:general}. 

\begin{prop}\label{mugrad}
Given $A \subseteq \Omega \subseteq \R^n$, with $A$ Borel and $\Omega$ open, given $\mu \in \calP(\R^n)$ which does not charge small sets, and $f : A \to \R^n$ measurable, we say that $f$ is $\mu$-differentiable on $A$ if there exists a countable family $(\phi_n)_{n \in \N}$, $\phi_n : \Omega \to \R$ such that
$$\forall x \in A, \exists n \in \N, \quad \phi_n(x) = f(x) \text{ and } \nabla \phi_n(x) \text{ exists}.$$
Then there exists a $\mu$-a.e. unique function $\nabla_\mu f : A \to \R^n$ enjoying the property
\begin{equation}\label{gradprop}
\begin{split}
\text{for all $\phi : \Omega \to \R$, $\nabla_\mu f$ coincides with $\nabla \phi$ on $\mu$-a.e. point $x$}\\ \text{such that $f(x) = \phi(x)$ and $\nabla \phi (x)$ exists.}
\end{split}
\end{equation}
\end{prop}

In order to prove the statement above, one can build $\nabla_\mu f$ as being equal to $\nabla \phi_n$ on almost all the set $\{ x : f(x) = \phi_n(x)\mbox{ and $\nabla \phi_n (x)$ exists}\}$, which is well-defined a.e. by Proposition \ref{isot}. The same proposition allows us to check easily that it satisfies the desired property, and is as such unique.
%
%\begin{proof}
%Let us denote for all $i, j$,
%\begin{align*}
%A_i &= \{ x : \phi_i(x) = f(x) \text{ and } \nabla \phi_i(x) \text{ exists}\}, &A_{i,j} &= A_i \cap A_j,\\
%B_{i,j} &= \isot_\mu (A_{i,j}), &N &= \bigcup_{i \neq j} A_{i,j} \setminus B_{i,j}.
%\end{align*}
%We follow the proof of Proposition \ref{gradientbar}.
%By Proposition \ref{isot}, for all $i \neq j$, $\nabla \phi_i(x) = \nabla \phi_j(x)$ for all $x \in B_{i,j}$. Hence we can define $\nabla_\mu f : A \setminus N \to \R^n$ by
%$$\nabla_\mu f(x) = \nabla \phi_i(x) \text{ if } x \in A_i,$$
%because $\bigcup_{i \in \N} A_i = A$. Since $N$ is $\mu$-negligible as a countable union of negligible sets, $\nabla_\mu f$ is defined on $\mu$ almost all $A$.
%
%The function $\nabla_\mu f$ satisfies the required property. Indeed, take $\phi : \Omega \to \R^n$, and denote
%\begin{align*}
%C_i &= \{x : f(x) = \phi_i(x) = \phi(x) \text{ and } \nabla \phi(x), \nabla \phi_i(x) \text{ exist}\}\\
%C &= \{ x : f(x) = \phi(x) \text{ and } \nabla \phi(x) \text{ exists} \}.
%\end{align*}
%Again by Proposition \ref{isot}, $\nabla \phi(x) = \nabla \phi_i(x)$ on almost all $C_i$, but $C_i \subseteq A_i$, where $\nabla \phi_i$ and $\nabla_\mu f$ coincide almost everywhere. It yields that $\nabla \phi(x) = \nabla_\mu f(x)$ on almost all $C_i$, hence on almost all $\bigcup_{i \in N} C_i = C$.
%\end{proof}

The connection between the two notions we defined is contained in the following statement.

\begin{prop}\label{mudiff_isot}
If $f : A \to \R$ is $\mu$-differentiable, then for a.e $x \in A$, $\nabla_\mu f(x)$ is a $\mu$-isotropic gradient of $f$.
\end{prop}

Together with Proposition \ref{criticpt}, this would allow us to follow the usual strategy described in Section \ref{sec:gene_facts} to prove the existence of an optimal $T$ based on the Kantorovich potential $\phi$ and $\nabla_\mu \phi$.

\bigskip

%change: accent and Capital
 {\bf Acknowledgments} This paper is part of the work of the ANR project ANR-12-BS01-0014-01 GEOMETRYA, whose support is gratefully acknowledged by the third author. The work started when the second author was master student at Paris-Sud, financed by {\it Fondation Math\'{e}matique Jacques Hadamard}, which is also gratefully acknowledged for its support.

\bigskip
\bigskip

{\small

\noindent\begin{minipage}{8cm}
Paul Pegon -- Filippo Santambrogio,\\
Laboratoire de Math\'ematiques d'Orsay,\\
 Universit\'e Paris-Sud,\\
 91405 Orsay cedex, FRANCE,\\
  {\tt paul.pegon@math.u-psud.fr}\\
    {\tt filippo.santambrogio@math.u-psud.fr}
    \end{minipage}
    %\bigskip
    %
\begin{minipage}{6cm}
Davide Piazzoli,\\
Cambridge Centre for Analysis,\\
University of Cambridge,\\
Wilberforce Road,\\
Cambridge, CB3 0WB,  UK,\\
{\tt D.Piazzoli@dpmms.cam.ac.uk}
  \end{minipage}
%  \bigskip
%  
%\noindent\begin{minipage}{7cm}
%Filippo Santambrogio,\\
%Laboratoire de Math\'ematiques d'Orsay,\\
% Universit\'e Paris-Sud,\\
% 91405 Orsay cedex, FRANCE,\\
%  {\tt filippo.santambrogio@math.u-psud.fr}
%    \end{minipage}
  
  }

\begin{thebibliography}{99}


\bibitem{AlbAmb}{\sc  G. Alberti and L. Ambrosio}, A geometrical approach to monotone
functions in $\R^n$, {\it Math. Z.}, 230 (1999), pp. 259--316.

\bibitem{Brenier91} {\sc Y. Brenier}, Polar factorization and monotone
rearrangement of vector-valued functions, \textit{Communications on Pure and
Applied Mathematics} 44, 375-417, 1991.

\bibitem{EvaGar} {\sc L. C. Evans and R. F. Gariepy}, {\it Measure theory and fine properties of
functions}, Studies in Advanced Mathematics, CRC Press, Boca Raton,
FL, 1992.

\bibitem{ChaDeP09} 
{\sc T. Champion and L. De Pascale},
The Monge problem in $R^d$,
{\it Duke Math. J.}
Volume 157, Number 3 (2011), 551-572. 
 
\bibitem{ChaDeP DCDS} 
{\sc T. Champion and L. De Pascale},
 On the twist condition and $c$-monotone transport plans, {\it Discr. Cont. Dyn. Syst. Ser. A}, Vol 34, No 4,   2014, 1339--1353
 
\bibitem{ChaDePJuu}  {\sc T. Champion, L. De Pascale and P. Juutinen,}
The $\infty$-Wasserstein distance: local solutions and existence of
   optimal transport maps,  {\it SIAM J. of Mathematical Analysis}, 40, (2008), no. 1, 1-20.
   
\bibitem{DelSalSob}  {\sc J. Delon, J. Salomon, A. Sobolevskii,}
    Local matching indicators for transport problems with concave costs,
{\it
SIAM J. Disc. Math}, 26 (2), pp. 801-827 (2012).
 
 \bibitem{Fed}
{\sc H. Federer}, {\it Geometric Measure Theory}, Classics in Mathematics, Springer, 1996 (reprint of the 1st ed. Berlin, Heidelberg, New York 1969 edition)
 
  \bibitem{GaMc} 
 {\sc W. Gangbo, R. McCann}, 
  The geometry of optimal
 transportation, 
{\it Acta Math.},  177 (1996), 113--161.

 \bibitem{Kan}
 {\sc L.V. Kantorovich, } On the translocation of masses, {\it C.R.
 (Dokl.) Acad. Sci. URSS},   37 (1942), 199-201.

 \bibitem{Kan1}
 {\sc L.V. Kantorovich, }, On a problem of Monge (in Russian),
{\it  Uspekhi Mat. Nauk.,} 3 (1948), 225-226.


\bibitem{MTW} {\sc X.-N. Ma, N. S. Trudinger, et X.-J. Wang} Regularity of potential functions of
the optimal transportation problem. {\it Arch. Ration. Mech. Anal.}, 177(2) :151--183, 
2005.



 \bibitem{Mon} 
 {\sc G. Monge,} 
M\'emoire sur la th\'eorie des D\'eblais et des Remblais, 
 {\it  Histoire de l'Acad\'emie des Sciences de Paris}, 1781.
 
 \bibitem{Pra} 
 {\sc A. Pratelli,} 
 On the sufficiency of c-cyclical monotonicity
 for optimality of transport plans, 
 {\it  Math. Z., } 258 (2008), no. 3, 677--690
 


\bibitem{villani} {\sc C. Villani}, {\it Topics in Optimal Transportation}. Graduate Studies in Mathematics, AMS, 2003.
\end{thebibliography}
\end{document}